\documentclass[12pt,a4paper]{amsart}
\usepackage[utf8]{inputenc}
\usepackage{fullpage}
\usepackage{amsmath,amsthm,comment}

\usepackage{amsfonts}
\newtheorem*{thm}{Theorem}
\newtheorem*{lem}{Lemma}
\newtheorem*{cor}{Corollary}

\theoremstyle{definition}
\newtheorem*{rem}{Remark}

\newcommand{\undd}[1]{\underline{\mathrm{d}}( #1 )}
\newcommand{\oved}[1]{\overline{\mathrm{d}}( #1 )}
\newcommand{\ovedf}[1]{{\mathrm{d}}^*( #1 )}
\newcommand{\dens}[1]{{\mathrm{d}}( #1 )}
\newcommand{\abs}[1]{\left\lvert #1 \right\rvert}
\def\F{\mathbb{F}}

\def\Z{\mathbb{Z}}
\def\N{\mathbb{N}}
\usepackage{amssymb}
\title[A simple proof of a Kneser-type theorem in $\sigma$-finite Abelian groups]{Kneser's Theorem in $\sigma$-finite Abelian groups}\author[P.-Y. Bienvenu]{Pierre-Yves Bienvenu}
\address{P.-Y. Bienvenu, Univ Lyon,  CNRS, ICJ UMR 5208, 
69622 Villeurbanne cedex, France}
\email{pierre.bienvenu@mpim-bonn.mpg.de}
\author[F. Hennecart]{Fran\c cois Hennecart}
\address{F. Hennecart, Univ Lyon, UJM-Saint-\'Etienne, CNRS, ICJ UMR 5208, 42023 Saint-\'Etienne, France}
\email{francois.hennecart@univ-st-etienne.fr}

\thanks{This work was performed within the framework of the LABEX MILYON (ANR-10-LABX-0070) of Universit\'e de Lyon, within the program ``Investissements d'Avenir" (ANR-11-IDEX-0007) operated by the French National Research Agency (ANR)}

\begin{document}
\begin{abstract}
Let $G$ be a $\sigma$-finite abelian group, i.e. $G=\bigcup_{n\geq 1} G_n$ where $(G_n)_{n\geq 1}$ is a non decreasing sequence of
finite subgroups.
For any $A\subset G$, let 
$\undd{A}:=\liminf_{n\to\infty}\frac{|A\cap G_n|}{|G_n|}$ be its lower asymptotic density.
We show that for any subsets $A$ and $B$ of $G$,
whenever $\undd{A+B}<\undd{A}+\undd{B}$,
the sumset $A+B$ must be periodic, that is, a union of translates
of a subgroup $H\leq G$ of finite index.
This is exactly analogous to Kneser's theorem regarding the density
of infinite sets of integers. Further,
we show similar statements for the upper asymptotic density in the case where $A=\pm B$.
An analagous statement had already been proven by Griesmer
in the very general context of countable abelian groups,
but the present paper provides a much simpler argument specifically tailored for the setting of $\sigma$-finite abelian groups.
This argument relies on an appeal to another theorem of Kneser, namely the one regarding finite sumsets in an abelian group. 
\end{abstract}
\maketitle

\section{Introduction and statement of the results}
Let $A$ be a subset of an abelian group $G$, $\mu$ be a measure on $G$ and $C>0$ be a real number.
Inverse results in additive number theory refer to those in which starting from the small doubling condition
$\mu(A+A)< C\mu(A)$, it is possible to deduce structural information on $A$ and its sumset $A+A$.
For finite sets $A$, the measure of $A$ can be its cardinality. More generally for infinite sets in a locally compact group,  $\mu(A)$ can be chosen as the Haar measure of $A$. For infinite sets in a discrete semigroup, we may use 
various notions of density instead of measure. 

One of the most popular inverse result is Kneser's theorem. 
In an abelian group with $\mu(\cdot)=|\cdot|$, the counting measure, and $C \le2$ it provides mainly a periodical structure for sumsets $A+B$ such that $|A+B|<|A|+|B|-1$, yielding also a partial structure for $A,B$ themselves.

In the particular semigroup $\mathbb{N}$ 
of positive integers
with $\mu=\underline{\mathrm{d}}$ being the lower asymptotic density, there exists again such a result
due to Kneser when $\underline{\mathrm{d}}(A+B)< \underline{\mathrm{d}}(A)+\underline{\mathrm{d}}(B)$ (see \cite{Kn1} or \cite{HR}); basically, $A+B$ is then a union of residue classes modulo some integer $q$. When $\mu$ is the upper asymptotic density or the upper Banach density, some structural information is still available (see \cite{Ji} and \cite{BJ} respectively).

Our goal is to investigate the validity of similar results in abelian $\sigma$-finite groups. A group $G$ is said to be
$\sigma$-finite if it is infinite and admits an \textit{exhausting sequence}, that is
a non decreasing sequence $(G_n)_{n\ge1}$ of finite subgroups such that
 $G=\bigcup_{n\ge1}G_n.$
As examples we have the polynomial ring
$\mathbb{F}_{p^r}[t]$ for any prime $p$ and integer $r\geq 1$.
More generally, let $(C_n)_{n\geq 1}$ be a sequence of finite groups, $G_\N=\prod_{n\in \N}C_n$ and $G_n=\prod_{i\leq n} C_i\leq G_\N$. Then $G=\bigcup_{n\geq 1} G_n$ is $\sigma$-finite.
Another class of $\sigma$-finite groups is given by  
the $p$-Prüfer groups
$\Z(p^\infty)=\bigcup_{r\ge1}\mathbb{U}_{p^{r}}$ where
$p$ is a prime number and $\mathbb{U}_{q}$ denotes the 
group of complex $q$-th roots of unity.
Further, if $(d_n)$ is a sequence of integers satisfying $d_n\mid d_{n+1}$ and $G_n=\mathbb{U}_{d_n}$, then $G=\bigcup_{n\geq 1} G_n$
 is $\sigma$-finite.

For any $A\subset G$, we define its lower and upper asymptotic densities as
$$
\undd{A}:=\liminf_{n\to\infty}\frac{|A\cap G_n|}{|G_n|}\quad\text{and}\quad
\oved{A}:=\limsup_{n\to\infty}\frac{|A\cap G_n|}{|G_n|},
$$
respectively. We observe that this definition implicitly requires  a particular exhausting sequence to be fixed.
If both limits coincide,  we denote by $\dens{A}$ their common value. 
This type of groups and densities were already studied in the
additive combinatorics literature;
Hamidoune and Rödseth \cite{HaR} proved that if $\langle A\rangle =G$
and $\alpha=\oved{A}>0$, then $hA=G$ for some $h=O(\alpha^{-1})$.
Hegyvári \cite{Heg} showed that 
then $h(A-A)=\langle A-A\rangle$ for some 
$h=O(\log \alpha^{-1})$, where again $\alpha=\oved{A}>0$; this was improved by Hegyvári and the second author \cite{HeHe}.

For a subset $X$ of an abelian group $G$,
 we denote by $\mathrm{Stab}_G(X)$ (or simply $\mathrm{Stab}(X)$ when no ambiguity is possible) the period or
 stabilizer of $X$, i.e. the subgroup $\{g\in G : \forall x\in X,\, g+x\in X\}$.

\begin{thm}\label{thm11}
Let $G$ be a $\sigma$-finite abelian group and $(G_n)_{n\geq 1}$ a fixed exhausting sequence. Let 
$A\subset G$ and $B\subset G$ satisfy
one of the following hypotheses:
\begin{enumerate}
\item $\undd{A+B}< \undd{A}+\undd{B}$;
\item $\oved{A+B}< \oved{A}+\oved{B}$ and $A=B$ or $A=-B$;
\item $\oved{A+B}< \oved{A}+\undd{B}$
\end{enumerate}
Then $H=\mathrm{Stab}(A+B)$ has finite index $q$ and admits a density equals to $q^{-1}$.
Moreover, the numbers $a,b,c$ of cosets of $H$ met by $A,B,A+B$ respectively satisfy
$c=a+b-1$.
\end{thm}
We will refer to the statements obtained with hypothesis 1,2 or 3
as Theorem 1,2 or 3 respectively.

Let $A,B,a,b,c$ be as in Theorem 1.
Then $A+B$ admits a density which is rational, namely $c/q$.
Further, let $\epsilon= 1-\frac{\dens{A+B}}{\undd{A}+\undd{B}}$.
Then we have
$$
(1-\epsilon)\frac{a+b}{q}\geq (1-\epsilon)(\undd{A}+\undd{B})
= \dens{A+B}=\frac{a+b-1}{q}.
$$
Therefore $a+b\leq \epsilon^{-1}$.
We also get $$q\leq ((\undd{A}+\undd{B})\epsilon)^{-1}=(\undd{A}+\undd{B}-\dens{A+B})^{-1}.$$

We point out that a slightly weaker conclusion was obtained under 
a similar hypothesis with the upper Banach density
in the general setting of countable abelian groups,
by Griesmer \cite{Griesmer}.
To state his theorem, recall that 
a F\o lner sequence for an additive group $G$ is a sequence $\Phi=(F_n)_{n\ge1}$ of finite subsets in $G$ such that 
$$
\lim_{n\to \infty}\frac{|(g+F_n)\Delta F_n|}{|F_n|}=0,\quad \text{for all }g\in G,
$$
where $\Delta$ is the symmetric difference operator. We let
$$
\ovedf{A}:=\sup_{\Phi}\limsup_{n\to\infty}\frac{|A\cap F_n|}{|F_n|}
$$
be the upper Banach density of $A$, where the supremum ranges over F\o lner sequences $\Phi$.
Under the hypothesis $\ovedf{A+B}< \ovedf{A}+\ovedf{B}$,
Griesmer obtained that $\ovedf{A+B}=\ovedf{A+B+H}$ for some subgroup $H$ of finite index.
The value of the present paper resides, besides the stronger conclusion that $A+B=A+B+H$, 
in the simplicity and brevity of
our elementary argument, in contrast to Griesmer's involved, ergodic-theoretic method. We make full use of the $\sigma$-finiteness 
property in order to  reduce the problem to passing to the limit
in Kneser's theorem for finite sets.
\begin{rem}
In our theorem, the lower asymptotic density may not be replaced by the upper asymptotic nor upper Banach density without any extra hypothesis.
Indeed, let $G$ be a $\sigma$-finite abelian group. Let $(G_n)_{n\geq 0}$ be an exhausting sequence of finite subgroups
such that $\abs{G_{n+1}}/\abs{G_n}\rightarrow\infty$. 
Let $A=\bigcup_{n\geq 1} G_{2n}\setminus G_{2n-1}$
and $B=\bigcup_{n\geq 0} G_{2n+1}\setminus G_{2n}$. Then $\oved{A}=\oved{B}=1$, 
therefore $\oved{A+B}=1$ and $\ovedf{A}=\ovedf{B}=\ovedf{A+B}=1$. Yet it is easy to see that $A+B=A\cup B=G\setminus G_0$ since $G_n\setminus G_{n-1}+G_m\setminus G_{m-1}=G_n\setminus G_{n-1}$ whenever $n>m$.
Hence the stabilizer of $A+B$ is $G_0$, which is a finite group, in particular it is not of finite index. We are thankful to the anonymous referee for this construction.

However, the fact that the extra hypothesis $B=\pm A$ suffices for the upper asymptotic density Kneser theorem is in contrast with the situation in the integers, cf \cite{Ji}.
Further, there exists a $\sigma$-finite group $G$ and a set $A\subset G$ such that
$\ovedf{A+A}<2\ovedf{A}$ but $\mathrm{Stab}(A+A)$ is finite.
Here is our construction. Let $G$ be a $\sigma$-finite abelian group.
Assume that $G$ admits  an exhausting sequence  of finite subgroups $(G_n)_{n\geq 0}$ satisfying
that none of the quotients $G_{n+1}/G_n$ have exponent 2; this is the case of $G=\F_3[t]$ for instance. 
For any $n\ge1$, let $x_n\in G_{n+1}\setminus G_n$ such that $2x_n\not\in G_n$. Let
$$
A=\bigcup_{n\ge1}(x_n+G_n).
$$
Then 
$$
A+A=\bigcup_{n\ge1}(\{x_n,2x_n\}+G_n).
$$
Since the sequence $(x_n+G_n)_{n\in\N}$ is a F\o lner sequence, we see that $\ovedf{A}=\ovedf{A+A}=1$
so $\ovedf{A+A}<2\ovedf{A}$.
Now we prove that $\mathrm{Stab}(A+A)=G_0$, a finite subgroup\footnote{We may even choose it to be trivial.}.
Indeed, let $x\in G\setminus G_0$, thus $x\in G_{k+1}\setminus G_{k}$ for some $k\ge0$. For any $g\in G_k$ we have
$
x+x_k+g\in G_{k+1}\setminus G_k.
$
Assume for a contradiction that 
$x+(\{x_k,2x_k\}+G_k)\subset A+A$. In particular,
$x+x_k+g\in A+A$ and  $x+2x_k+g\in A+A$ for some $g\in G_k$. Then $x+x_k+g=x_k+g'$ or $x+x_k+g=2x_k+g'$, 
and $x+2x_k+g=2x_k+g''$ or $x+2x_k+g=x_k+g''$
for some  $g',g''\in G_k$. Since $x\not\in G_k$, it follows that 
$x+x_k+g=2x_k+g'$ and $x+2x_k+g=x_k+g''$.
Therefore,
$x=x_k+g'-g\in A$ and $x=-x_k+g''-g\in -A$. But $A\cap(-A)=\emptyset$, whence the contradiction. We infer that $x\not\in \mathrm{Stab}(A+A)$. Thus
$$
\mathrm{Stab}(A+A)\le G_0.
$$
It is easy to check that $G_0\le \mathrm{Stab}(A+A)$ thus $\mathrm{Stab}(A+A) = G_0$ as announced.
This remark shows that even the context of $\sigma$-finite groups,
a small doubling in Banach density of a set $A$ does not imply that $A+A$ is periodic, thus Griesmer's 
weaker structural conclusion is optimal.
\end{rem}
\begin{rem}
If $G$ has no proper subgroup of finite index, which is the case of 
$\Z(p^\infty)$ for instance,
the conclusion implies that $A+B=G$, 
so
$\undd{A+B}<\undd{A}+\undd{B}$ is impossible except in the trivial case
$\undd{A}+\undd{B}>1$.
\end{rem}

\section{Proofs}

We detail the proof of Theorem \ref{thm11}. Theorems 2 and 3 will be deduced straightforwardly by the same arguments.

\begin{proof}[Proof of Theorem \ref{thm11}]
Let $\alpha:=\undd{A}$ and $\beta:=\undd{B}$.
Let us assume that 
$\undd{A+B}<\alpha+\beta$. Therefore $\alpha+\beta>0$ and
there exists $\epsilon >0$ such that
$$
\undd{A+B}<(1-\epsilon)(\alpha+\beta).
$$
By definition of the lower limit, there exists an increasing sequence
$(n_i)_{i\geq 1}$ of integers such that
$$
|(A+B)\cap G_{n_i}|<(1-\epsilon/2)|G_{n_i}|(\alpha+\beta).
$$
Furthermore,
denoting $A_n=A\cap G_n$ and $B_n=B\cap G_n$,
 for any $i$ large enough, we have
$$|A_{n_i}|+|B_{n_i}|>(\alpha+\beta) |G_{n_i}|(1-\epsilon/2)/(1-\epsilon/3).$$
Combining these two bounds together with the fact that $A_{n_i}+B_{n_i}\subset (A+B)\cap G_{n_i}$, 
and reindexing the sequence $(n_i)$ if necessary, we have
\begin{equation}
\label{realstart}
|A_{n_i}+B_{n_i}|<(1-\epsilon/3)(|A_{n_i}|+|B_{n_i}|)
\end{equation}
for all $i\geq 1$. In particular $|A_{n_i}+B_{n_i}|<|A_{n_i}|+|B_{n_i}|-1$ for $i$ large enough (by reindexing if necessary, for any $i\geq 1$),
since $|A_{n_i}|+|B_{n_i}|\rightarrow \infty$ as a consequence of $\alpha+\beta>0$.

We are now able to apply
 Kneser's theorem in the finite abelian group $G_{n_i}$.
Let  $H_i=\mathrm{Stab}_{G_{n_i}}(A_{n_i}+B_{n_i})$.
We obtain that $H_i\neq \{0\}$ for each $i\geq 1$ and
letting $a_i,b_i,c_i$ be the number of cosets of $H_i$ met by $A_{n_i},B_{n_i},A_{n_i}+B_{n_i}$
respectively, we get
$c_i=a_i+b_i-1$. Reformulating inequality \eqref{realstart} in terms of these quantities yields
\begin{align*}
(1-\epsilon/3)(a_i+b_i)|H_i| 
&\ge (1-\epsilon/3)
 (|A_{n_i}|+|B_{n_i}|)\\
 &>|A_{n_i}+B_{n_i}|=
 (a_{i}+b_{i}-1)|H_{i}|
\end{align*}
from which we infer that $a_{i}+b_{i}<3/\epsilon$. Moreover,
$$
(\alpha+\beta)|G_{n_i}|/2<|A_{n_i}|+|B_{n_i}|\le  (a_{i}+b_{i})|H_{i}|
$$
so $[G_{n_i}:H_i]=|G_{n_i}|/|H_i|<\frac{a_i+b_i}{\alpha+\beta}<
6/((\alpha+\beta)\epsilon)$. 
By the pigeonhole principle,
upon extracting again a suitable subsequence of $(n_i)$, one may assume that 
$[G_{n_i}:H_i]=k$ for any $i\geq 1$ and some fixed $k<6/((\alpha+\beta)\epsilon)$.

Let us set
$$
\mathcal{A}_i=\{K< G_{n_i}\,:\, [G_{n_i}:K]=k\}.
$$
Thus $H_{i}\in  \mathcal{A}_{i}$ for any $n\geq 1$. 

%

\begin{lem}\label{lem2}
Let $i\geq 1$.
If $L\in\mathcal{A}_{i+1}$, there exists
$K\in\mathcal{A}_{i}$ such that $K\subset L$.
\end{lem}
\begin{proof}
Let $L\in\mathcal{A}_{i+1}$. Let us set $K'=L\cap G_{n_i}$. Thus
$$
G_{n_i}/K'\simeq (L+G_{n_i})/L\leq G_{n_{i+1}}/L
$$
so $[G_{n_i}:K']$ divides $k$. 
Let us write
$$
|G_{n}|=kg,\ [G_{n_i}:K']=\frac{k}{h},\ |K'|=gh
$$
for some positive integers $g$ and $h$.
Since $K'$ is abelian and $h$ divides $ |K'|$,
there exists a subgroup $K$ of $K'$ of index $h$. It satisfies $$[G_{n_i}:K]=[G_{n_i}:K']\times [K':K]=k.\qedhere
$$
\end{proof}
We draw from the lemma the following corollary
by an easy induction. 
\begin{cor}
\label{rec}
For all pairs of integers  $j>i \geq 1$, there exists a sequence of subgroups $K_\ell\in \mathcal{A}_\ell$ for $\ell\in [i,j]$ such that $K_\ell\leq K_{\ell+1}$ for all
$\ell\in [i,j)$ and $K_j=H_j$.
\end{cor}
Borrowing terminology from   graph theory,
for 
$K\in \mathcal{A}_i$ and
$L\in\mathcal{A}_j$ where $i\leq j$, we call a \textit{path} from $K$ to  $L$ any non decreasing sequence of subgroups $K_\ell\in\mathcal{A}_\ell,\ell\in [i,j]$ for which $K_i=K$ and $K_j=L$.
With this terminology, the conclusion of the corollary 
is that there exists a path from $K_i$ to $H_j$.

We shall construct inductively a non decreasing subsequence  of subgroups $K_i\in \mathcal{A}_{i}$
for $i\geq 1$ such that for any $i$, the set  of integers $j\geq i$ for which there exists a path from
$K_i$ to  $H_j$ is infinite.

To construct $K_1\in\mathcal{A}_1$ with the desired property, let us observe that $\mathcal{A}_1$ is finite and  invoke  the corollary 
and the pigeonhole principle.

Suppose that $K_1,\ldots,K_i$ are already constructed for some $i\geq 1$; one constructs 
$K_{i+1}$ by observing again that $\mathcal{A}_{i+1}$ is finite and applying the pigeonhole principle. Indeed, there 
are infinitely many $j>i$ such that a path from $K_i$ to $H_j$ exists, but only finitely many $K\in \mathcal{A}_{i+1}$ through
which these paths may pass, so there exists $K\in \mathcal{A}_{i+1}$ having the property that there exist 
infinitely many $j> i$ such that a path from $K_i$ to $H_j$ over $K\in \mathcal{A}_{i+1}$ exists.
One then selects $K_{i+1}$ to be such a subgroup $K$.

The sequence $(K_i)_{i\geq 1}$ being non decreasing,
the union
$$
K=\bigcup_{i\ge 1}K_{i}
$$
is a subgroup of $G$.
In fact, $K\leq\mathrm{Stab}(A+B)=:H$; indeed, let $g\in K$
and
$x=a+b\in A+B$ where $(a,b)\in A\times B$.
Let $i$ satisfy $g\in K_i$ and $x\in A_{n_i}+B_{n_i}$.
Since $K_i$ is included in  $H_j$ for infinitely many $j\geq 1$, there exists in particular some $j\geq i$ for which $K_i\leq H_j$. Thus $x+g\in A_{n_j}+B_{n_j}$
so $g\in H$.

Since $K_{i}\subset H\cap G_{n_i}$, the indices $[G_{n_i}:H\cap G_{n_i}]$ divide
 $k=[G_{n_i}:K_{i}]$. Now
$$
G_{n_i}/(H\cap G_{n_i})\simeq(G_{n_i}+H)/H,
$$
and $((G_{n_i}+H)/H)_{i\in\N}$ is a non decreasing sequence of subgroups of bounded
indices of $G/H$, whose union is $G/H$.
It is therefore stationary, so  $(G_{n_i}+H)/H=G/H$
for any large enough $i$.
 It follows that
$[G:H]$ divides $k$,
in particular $ [G:H]\le k <6/((\alpha+\beta)\epsilon).$


Furthermore, if $i$ is large enough, we have
$$
\frac{\abs{H\cap G_{n_i}}}{|G_{n_i}|}= \abs{(G_{n_i}+H)/H}^{-1}=\abs{G/H}^{-1}
$$
whence we deduce that $\dens{H}$ exists and equals $q^{-1}=[G:H]^{-1}$.

Let $C,D$ and $F$ be the projections of $A,B$ and $A+B$, respectively, in $G/H$. Then $F=C+D$.
In view of the previous paragraph and the inequality $\undd{A+B}<\undd{A}+\undd{B}$,
we obtain that $|C+D|<|C|+|D|$.
Now $|C+D|\geq |C|+|D|-1$, since otherwise by Kneser's theorem applied in the finite  abelian group $G/H$, the set
$F$ would admit a non trivial period and $\mathrm{Stab}(A+B)$ would be strictly larger than $H$. Therefore $|C+D|= |C|+|D|-1$, 
and we conclude.
\end{proof}

We now prove the same result with the upper asymptotic density in place of the lower one in the case $B=\pm A$, that is, Theorem 2.
Let $\alpha:=\oved{A}=\oved{B}=\beta$ and suppose that 
$\oved{A+B}<2\alpha(1-\epsilon)$ for some $\epsilon >0$.
Then on the one hand we have for any large enough $n$
$$
|A_n+B_n|\le |(A+B)\cap G_n|<2\alpha(1-\epsilon/2)|G_n|.
$$
On the other hand, since $\abs{A_n}=\abs{B_n}$ for all $n$,  there exist infinitely many $n$ such that 
\begin{equation}\label{fails}
|A_n|+|B_n|>\frac{1-\epsilon/2}{1-\epsilon/3}(\alpha+\beta)|G_n|.
\end{equation}
Hence equation \eqref{realstart} holds along some increasing sequence $(n_i)$.
From there we are in position to conclude as in the proof of Theorem \ref{thm11}.
The reader may observe that equation \eqref{fails} fails for general $A$ and $B$.

Finally we prove Theorem 3.
Let $\alpha:=\oved{A}$ and $\beta=\undd{B}$ and suppose that 
$\oved{A+B}<(\alpha+\beta)(1-\epsilon)$ for some $\epsilon >0$.
Then on the one hand we have for any large enough $n$
$$
|A_n+B_n|\le |(A+B)\cap G_n|<(\alpha+\beta)(1-\epsilon/2)|G_n|
$$
and
$$
|B_n|>\frac{1-\epsilon/2}{1-\epsilon/3}\beta|G_n|.
$$
On the other hand
there exist infinitely many $n$ such that 
\begin{equation*}
|A_n|>\frac{1-\epsilon/2}{1-\epsilon/3}\alpha|G_n|.
\end{equation*}
Hence again  we recover equation \eqref{realstart} along some subsequence and can conclude.

\end{document}